\documentclass{article}%
\usepackage{amsfonts}
\usepackage{amssymb}
\usepackage{graphicx}
\usepackage{amsmath}%
\newtheorem{theorem}{Theorem}

\newtheorem{corollary}[theorem]{Corollary}

\newtheorem{definition}[theorem]{Definition}
\newtheorem{example}[theorem]{Example}

\newtheorem{lemma}[theorem]{Lemma}

\newtheorem{proposition}[theorem]{Proposition}

\newenvironment{proof}[1][Proof]{\textbf{#1.} }{\ \rule{0.5em}{0.5em}}
\begin{document}

\title{On the linear algebra of local complementation}
\author{Lorenzo Traldi\\Lafayette College\\Easton, Pennsylvania 18042}
\date{}
\maketitle

\begin{abstract}
We explore the connections between the linear algebra of symmetric matrices
over $GF(2)$ and the circuit theory of 4-regular graphs. In particular, we
show that the equivalence relation on simple graphs generated by local
complementation can also be generated by an operation defined using inverse matrices.

\bigskip


\bigskip

Mathematics Subject\ Classification. 05C50

\end{abstract}

\section{Introduction}

This paper is about the connection between the circuit theory of 4-regular
multigraphs and the elementary linear algebra of symmetric matrices over the
two-element field $GF(2)$.

\begin{definition}
A square matrix $S=(s_{ij})$ with entries in $GF(2)$\ is \emph{symmetric} if
$s_{ij}=s_{ji}$ for all $i\neq j$. $S$ is \emph{zero-diagonal} if $s_{ii}=0$
for all $i$.
\end{definition}

As $GF(2)$ is the only field that concerns us, we will often omit the phrase
\textquotedblleft with entries in $GF(2)$.\textquotedblright

Symmetric matrices are important to us because they arise as adjacency
matrices of graphs, and for our purposes it is not important if the vertices
of a graph are listed in any particular order. That is, if $V$ is a finite set
then we regard a $V\times V$ matrix as a function $V^{2}\rightarrow GF(2)$. Of
course we must impose an order on $V$ in order to display a matrix.

Many zero-diagonal symmetric matrices are singular, and consequently do not
have inverses in the usual sense. Nevertheless matrix inversion gives rise to
an interesting relation among zero-diagonal symmetric matrices\textit{.}

\begin{definition}
\label{minv} Let $S$ be a zero-diagonal symmetric matrix. A \emph{modified
inverse} of $S$ is a zero-diagonal symmetric matrix obtained as follows: first
toggle some diagonal entries of $S$ to obtain an invertible symmetric matrix
$S^{\prime}$, and then toggle every nonzero diagonal entry of $(S^{\prime
})^{-1}$.
\end{definition}

Here \emph{toggling} refers to the function $x\mapsto x+1$, which interchanges
the elements of $GF(2)$.

The relation defined by modified inversion is obviously symmetric, but
examples indicate that it is not reflexive or transitive.

\begin{example}
\label{exone}\emph{For each of these four matrices, the set of modified
inverses consists of the other three.}%
\[%
\begin{pmatrix}
0 & 1 & 1\\
1 & 0 & 1\\
1 & 1 & 0
\end{pmatrix}
\text{ \ }%
\begin{pmatrix}
0 & 1 & 1\\
1 & 0 & 0\\
1 & 0 & 0
\end{pmatrix}
\text{ }%
\begin{pmatrix}
0 & 1 & 0\\
1 & 0 & 1\\
0 & 1 & 0
\end{pmatrix}
\text{ }%
\begin{pmatrix}
0 & 0 & 1\\
0 & 0 & 1\\
1 & 1 & 0
\end{pmatrix}
\]

\end{example}

\begin{example}
\emph{The modified inverses of the first of the following three matrices
include the other two, but not itself. The modified inverses of the second
include the first, but not itself or the third. The modified inverses of the
third include the first and itself, but not the second.}%
\[
\text{\ }%
\begin{pmatrix}
0 & 0 & 1 & 0\\
0 & 0 & 0 & 1\\
1 & 0 & 0 & 1\\
0 & 1 & 1 & 0
\end{pmatrix}
\text{ \ }%
\begin{pmatrix}
0 & 1 & 1 & 0\\
1 & 0 & 1 & 1\\
1 & 1 & 0 & 0\\
0 & 1 & 0 & 0
\end{pmatrix}
\text{ \ }%
\begin{pmatrix}
0 & 1 & 1 & 1\\
1 & 0 & 1 & 0\\
1 & 1 & 0 & 1\\
1 & 0 & 1 & 0
\end{pmatrix}
\]

\end{example}

Definition \ref{minv} yields an equivalence relation in the usual way.

\begin{definition}
Let $S$ and $T$ be zero-diagonal symmetric $GF(2)$-matrices. Then $S\sim
_{mi}T$ if $T$ can be obtained from $S$ through a finite (possibly empty)
sequence of modified inversions.
\end{definition}

We recall some relevant definitions from graph theory. In a 4-regular
multigraph every vertex is of degree 4. Loops and parallel edges are allowed;
a loop contributes twice to the degree of the incident vertex. In order to
distinguish between the two orientations of a loop it is technically necessary
to consider \emph{half-edges} rather than edges; we will often leave it to the
reader to refine statements regarding edges accordingly. A \emph{walk} in a
4-regular graph is a sequence $v_{1},h_{1},h_{2}^{\prime},v_{2},h_{2}%
,h_{3}^{\prime},...,h_{k-1},h_{k}^{\prime},v_{k}$ such that for each $i$,
$h_{i}$ and $h_{i}^{\prime}$ are distinct half-edges incident on $v_{i}$, and
$h_{i}$ and $h_{i+1}^{\prime}$ are half-edges of a single edge. The walk is
\emph{closed} if $v_{1}=v_{k}$. A walk in which no edge is repeated is a
\emph{trail}, and a closed trail is a \emph{circuit}. An \emph{Euler circuit}
is a circuit that contains every edge of the graph. Every connected 4-regular
multigraph has Euler circuits, and every 4-regular multigraph has\ \emph{Euler
systems}, each of which contains one Euler circuit for every connected
component of the graph.%

\begin{figure}
[ptb]
\begin{center}
\includegraphics[
trim=0.859233in 7.886825in 1.147843in 1.009787in,
height=1.2886in,
width=4.3932in
]%
{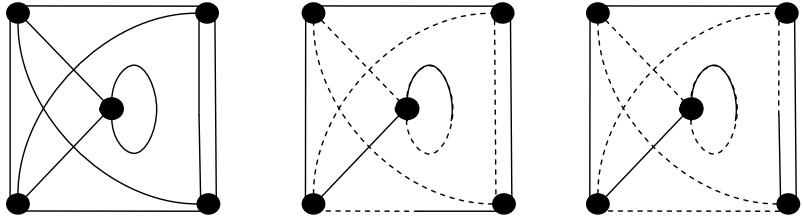}%
\caption{A 4-regular multigraph with two of its Euler circuits. To trace an
Euler circuit, maintain the dash pattern while traversing each vertex. (The
dash pattern may change within an edge, though.)}%
\label{intext5}%
\end{center}
\end{figure}

For example, Figure \ref{intext5} illustrates two Euler circuits in a
connected 4-regular multigraph. These two\ Euler circuits illustrate the following.

\begin{definition}
\label{kappa} If $C$ is an Euler system of a 4-regular multigraph $F$ and
$v\in V(F)$ then the $\kappa$\emph{-transform} $C\ast v$ is the Euler system
obtained by reversing one of the two $v$-to-$v$ trails within the circuit of
$C$ incident at $v$.
\end{definition}

The $\kappa$-transformations were introduced by Kotzig \cite{K}, who proved
the fundamental fact of the circuit theory of 4-regular multigraphs.

\begin{theorem}
(\textit{Kotzig's theorem}) All the Euler systems of a 4-regular multigraph
can be obtained from any one by applying finite sequences of $\kappa$-transformations.
\end{theorem}

In this paper our attention is focused on 4-regular multigraphs, but we should
certainly mention that the reader interested in general Eulerian multigraphs
will find Fleischner's books \cite{F1, F2} uniquely valuable. In particular,
Theorem VII.5 of \cite{F1} generalizes Kotzig's theorem to arbitrary Eulerian graphs.

The \emph{alternance} or \emph{interlacement} graph $\mathcal{I}(C)$
associated to an Euler system $C$ of a 4-regular multigraph $F$ was introduced
shortly after Kotzig's work became known \cite{Bold, CL, RR}. Two vertices
$v\neq w$ of $F$ are \textit{interlaced} with respect to $C$ if and only if
they appear in the order $v...w...v...w$ on one of the circuits included in
$C$.

\begin{definition}
The \emph{interlacement graph} of a 4-regular graph $F$ with respect to
an\ Euler system $C$ is the simple graph $\mathcal{I}(F,C)$ with
$V(\mathcal{I}(F,C))=V(F)$ and $E(\mathcal{I}(F,C))=\{vw$
$\vert$
$v$ and $w$ are interlaced with respect to $C\}$. The \emph{interlacement
matrix} of $F$ with respect to $C$ is the adjacency matrix of this graph; when
no confusion can arise, we use $\mathcal{I}(C)$ to denote both the graph and
the matrix.
\end{definition}

A simple graph that can be realized as an interlacement graph is called a
\emph{circle graph}, and Kotzig's $\kappa$-transformations give rise to the
fundamental operation of the theory of circle graphs, which we call
\emph{simple local complementation}. This operation has been studied by
Bouchet \cite{Bu, Bec, Bco}, de Fraysseix \cite{F}, and Read and Rosenstiehl
\cite{RR}, among others. The reader can easily verify that the effect of a
$\kappa$-transformation on an interlacement matrix is described as follows.

\begin{definition}
\label{lc} Let $S$ be a symmetric $n\times n$ matrix, and suppose $1\leq i\leq
n$. Then the \emph{simple local complement of }$S$ \emph{at} $i$ is the
symmetric $GF(2)$-matrix $S^{i}$ obtained from $S$ as follows: whenever $i\neq
j\neq k\neq i$ and $s_{ij}\neq0\neq s_{ik}$, toggle $s_{jk}$.
\end{definition}

We call this operation \emph{simple} local complementation to distinguish it
from the similar operation that Arratia, Bollob\'{a}s and Sorkin called
\textit{local complementation} in \cite{A1, A2, A}. The two operations differ
on the diagonal:

\begin{definition}
\label{nslc}Let $S$ be a symmetric $n\times n$ matrix, and suppose $1\leq
i\leq n$. Then the \emph{(non-simple) local complement of }$S$ \emph{at} $i$
is the symmetric matrix $S_{ns}^{i}$ obtained from $S^{i}$ as follows:
whenever $j\neq i$ and $s_{ij}\neq0$, toggle $s_{jj}$.
\end{definition}

The first $3\times3$ matrix of Example \ref{exone} has three distinct simple
local complements, which are the same as its three modified inverses. Each of
the three other $3\times3$ matrices of Example \ref{exone} has only two
distinct simple local complements, itself and the first matrix.

\begin{definition}
Let $S$ and $T$ be zero-diagonal symmetric $GF(2)$-matrices. Then $S\sim
_{lc}T$ if $T$ can be obtained from $S$ through a finite (possibly empty)
sequence of simple local complementations.
\end{definition}

\begin{proposition}
This defines an equivalence relation.
\end{proposition}

\begin{proof}
As $(S^{i})^{i}=S$, $\sim_{lc}$ is symmetric. The reflexive and transitive
properties are obvious.
\end{proof}

In Section 3 we prove a surprising result:

\begin{theorem}
\label{lcinv} Let $S$ and $T$ be zero-diagonal symmetric $GF(2)$-matrices.
Then $S\sim_{lc}T$ if and only if $S\sim_{mi}T$.
\end{theorem}

We might say that modified inversion constitutes a kind of \emph{global
complementation} of a zero-diagonal symmetric matrix (or equivalently, a
simple graph). Theorem \ref{lcinv} shows that even though individual global
complementations do not generally have the same effect as individual simple
local complementations, the two operations generate the same equivalence relation.

Theorem \ref{lcinv} developed as we read the work of several authors who have
written about the equivalence relation on looped graphs (or equivalently,
symmetric $GF(2)$-matrices) generated by (non-simple) local complementations
at looped vertices and pivots on unlooped edges; we denote this relation
$\sim_{piv}$. (If $G$ is a looped graph then a \emph{pivot} on an unlooped
edge $ab$ of $G$ is the triple simple local complement $((G^{a})^{b}%
)^{a}=((G^{b})^{a})^{b}$.) Although $\sim_{piv}$ is defined for symmetric
matrices and $\sim_{lc}$ is defined only for zero-diagonal symmetric matrices,
it is reasonable to regard $\sim_{lc}$ as a coarser version of $\sim_{piv}$,
obtained by ignoring the difference between looped and unlooped vertices.
Genest \cite{G1, G} called the equivalence classes under $\sim_{piv}$\emph{
Sabidussi orbits}. Glantz and Pelillo \cite{GP} and Brijder and Hoogeboom
\cite{BH, BH1, BH2} observed that another way to generate $\sim_{piv}$ is to
use a matrix operation related to inversion, the \emph{principal pivot
transform} \cite{Ts, Tu}. In particular, Theorem 24 of \cite{BH1} shows that
the combination of loop-toggling with $\sim_{piv}$ yields a description of
$\sim_{lc}$ that is different from Definition \ref{lc} (and also Definition
\ref{minv}). Ilyutko \cite{I1, I} has also studied inverse matrices and the
equivalence relation $\sim_{piv}$; he used them to compare the adjacency
matrices of certain kinds of chord diagrams that arise from knot diagrams.
Ilyutko's account includes an analysis of the effect of the Reidemeister moves
of knot theory, and also includes the idea of generating an equivalence
relation on nonsingular symmetric matrices by toggling diagonal entries.
Considering the themes shared by these results, it seemed natural to wonder
whether the connection between matrix inversion and $\sim_{piv}$ reflects a
connection between matrix inversion and the coarser equivalence relation
$\sim_{lc}$.

Theorem \ref{lcinv} is part of a very pretty theory tying the elementary
linear algebra of symmetric $GF(2)$-matrices to the circuit theory of
4-regular multigraphs. This theory has been explored by several authors over
the last forty years, but the relevant literature is fragmented and it does
not seem that the generality and simplicity of the theory are fully
appreciated. We proceed to give an account.

At each vertex of a 4-regular multigraph there are three \emph{transitions} --
three distinct ways to sort the four incident half-edges into two disjoint
pairs. Kotzig \cite{K} introduced this notion, and observed that each of the
$3^{n}$ ways to choose one transition at each vertex yields a partition of
$E(F)$ into edge-disjoint circuits; such partitions are called \emph{circuit
partitions} \cite{A1, A2, A} or \emph{Eulerian partitions} \cite{L, Ma}.
(Kotzig actually used the term \emph{transition} in a slightly different way,
to refer to only one pair of half-edges. As we have no reason to ever consider
a single pair of half-edges without also considering the complementary pair,
we follow the usage of Ellis-Monaghan and Sarmiento \cite{E} and Jaeger
\cite{J} rather than Kotzig's.)

In \cite{Ti} we introduced the following way to label the three transitions at
$v$ with respect to a given Euler system $C$. Choose either of the two
orientations of the circuit of $C$ incident at $v$, and use $\phi$ to label
the transition followed by this circuit; use $\chi$ to label the other
transition in which in-directed edges are paired with out-directed edges; and
use $\psi$ to label the transition in which the two in-directed edges are
paired with each other, and the two out-directed edges are paired with each
other. See Figure \ref{intext3}, where the pairings of half-edges are
indicated by using solid line segments for one pair and dashed line segments
for the other.%

\begin{figure}
[ptb]
\begin{center}
\includegraphics[
trim=1.143720in 8.170293in 1.866895in 1.293255in,
height=0.8925in,
width=3.6919in
]%
{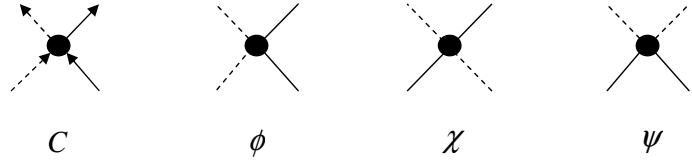}%
\caption{The three transitions at a vertex are labeled according to their
relationships with the incident circuit of $C$.}%
\label{intext3}%
\end{center}
\end{figure}

\begin{definition}
\label{relint} Let $C$ be an Euler system of a 4-regular multigraph $F$, and
let $P$ be a circuit partition of $F$. The \emph{relative interlacement matrix
}$\mathcal{I}_{P}(C)$ of $P$ with respect to $C$ is obtained from
$\mathcal{I}(C)$ by modifying the row and column corresponding to each vertex
at which $P$ does not involve the transition labeled $\chi$ with respect to
$C$:

(i) If $P$ involves the $\phi$ transition at $v$, then modify the row and
column of $\mathcal{I}(C)$ corresponding to $v$ by changing every nonzero
entry to $0$, and changing the diagonal entry from 0 to 1.

(ii) If $P$ involves the $\psi$ transition at $v$, then modify the row and
column of $\mathcal{I}(C)$ corresponding to $v$ by changing the diagonal entry
from 0 to 1.
\end{definition}

The relative interlacement matrix determines the number of circuits included
in $P$:%
\[
\nu(\mathcal{I}_{P}(C))+c(F)=\left\vert P\right\vert ,
\]
where $\nu(\mathcal{I}_{P}(C))$ denotes the $GF(2)$-nullity of $\mathcal{I}%
_{P}(C)$ and $c(F)$ denotes the number of connected components in $F$. We
refer to this equation as the \textit{circuit-nullity formula} or the
\textit{extended Cohn-Lempel equality}; many special cases and reformulations
have appeared over the years \cite{Be, BM, Bu, Br, CL, J1, Jo, KR, Lau, MP,
Me, M, R, So, S, Z}. A detailed account is given in \cite{Tbn}.

The circuit-nullity formula is usually stated in an equivalent form, with part
(i) of Definition \ref{relint} replaced by:

(i)$^{\prime}$ If $P$ involves the $\phi$ transition at $v$, then remove the
row and column of $\mathcal{I}(C)$ corresponding to $v$.

We use (i) instead because it is more convenient for the sharper form of the
circuit-nullity formula given in Theorem \ref{nullspaces}, which includes a
precise description of the nullspace of $\mathcal{I}_{P}(C)$.

\begin{definition}
\label{relcore}Let $P$ be a circuit partition of the 4-regular multigraph $F$,
and let $C$ be an Euler system of $F$. For each circuit $\gamma\in P$, let the
\emph{relative core vector} of $\gamma$ with respect to $C$ be the vector
$\rho(\gamma,C)\in GF(2)^{V(F)}$ whose nonzero entries correspond to the
vertices of $F$ at which $P$ involves either the $\chi$ or the $\psi$
transition, and $\gamma$ is singly incident. (A circuit $\gamma$ is
\emph{singly incident} at $v$ if $\gamma$ includes precisely two of the four
half-edges at $v$.)
\end{definition}

\begin{theorem}
\label{nullspaces}Let $P$ be a circuit partition of the 4-regular multigraph
$F$, and let $C$ be an Euler system of $F$.

(i) The nullspace of the relative interlacement matrix $\mathcal{I}_{P}%
(C)$\ is spanned by the relative core vectors of the circuits of $P$.

(ii) For each connected component of $F$, the relative core vectors of the
incident circuits of $P$ sum to $0$.

(iii) If $Q\subset P$ and there is no connected component of $F$ for which $Q$
contains every incident circuit of $P$, then the relative core vectors of the
circuits of $Q$ are linearly independent.
\end{theorem}

Theorem \ref{nullspaces} is an example of a comment made above, that the
generality and simplicity of the relationship between linear algebra and the
circuit theory of 4-regular multigraphs have not been fully appreciated.

On the one hand, Theorem \ref{nullspaces} is general: it applies to every
4-regular multigraph $F$, every Euler system $C$ and every circuit partition
$P$. In Proposition 4 of \cite{J1}, Jaeger proved an equivalent version of the
special case of Theorem \ref{nullspaces} involving the additional assumptions
that $F$ is connected and $C$ and $P$ involve different transitions at every
vertex. (These assumptions are implicit in Jaeger's use of left-right walks on
chord diagrams.) It is Jaeger who introduced the term \emph{core vector}; we
use \emph{relative core vector} to reflect the fact that in Definition
\ref{relcore}, the vector is adjusted according to the Euler system with
respect to which it is defined. We provide a direct proof of Theorem
\ref{nullspaces} in Section 4 below, but the reader who is already familiar
with Jaeger's special case may deduce the general result using
\emph{detachment} \cite{F1, N1} along $\phi$ transitions. Bouchet \cite{Bu}
gave a different proof with the additional restriction that $C$ and $P$
respect a fixed choice of edge-direction in $F$; or equivalently, that $P$
involve only transitions that are labeled $\chi$ with respect to $C$.
(Bouchet's result is also presented in \cite{GR}.) Genest \cite{G1, G} did not
use techniques of linear algebra, but he did discuss the use of bicolored
graphs to represent a circuit partition $P$ using an Euler system that does
not share a transition with $P$ at any vertex. These previous results may give
the erroneous impression that an Euler system provides information about only
those circuit partitions that disagree with it at every vertex. In fact, every
Euler system gives rise to a mapping%
\[
\{\text{circuit partitions of }F\}\rightarrow\{\text{subspaces of
}GF(2)^{V(F)}\}
\]
under which the image of an arbitrary circuit partition $P$ is the
$(\left\vert P\right\vert -c(F))$-dimensional subspace spanned by the relative
core vectors of the circuits in $P$.

On the other hand, Theorem \ref{nullspaces} is simple: Definition \ref{relint}
gives an explicit description of the relative interlacement matrix associated
to a circuit partition, and Definition \ref{relcore} gives an explicit
description of its nullspace. Although some version of Theorem
\ref{nullspaces} may be implicit in the theory of delta-matroids, isotropic
systems and multimatroids associated with 4-regular graphs (cf. for instance
\cite{B1, B2, B3, B4}), these structures are sufficiently abstract that it is
difficult to extract explicit descriptions like Definition \ref{relint} and
Definition \ref{relcore} from them.

The circuit-nullity formula implies that if $C$ is an Euler system of $F$,
then we can find every other Euler system $C^{\prime}$ of $F$ by finding every
way to obtain a nonsingular matrix from $\mathcal{I}(C)$ using the
modifications given in Definition \ref{relint}. Moreover, there is a striking
symmetry tying together the relative interlacement matrices of two Euler systems:

\begin{theorem}
\label{intinv}Let $F$ be a 4-regular graph with Euler systems $C$ and
$C^{\prime}$. Then%
\[
\mathcal{I}_{C^{\prime}}(C)^{-1}=\mathcal{I}_{C}(C^{\prime}).
\]

\end{theorem}

Like Theorem \ref{nullspaces}, Theorem \ref{intinv} generalizes results of
Bouchet and Jaeger discussed in \cite{Bu, GR, J1}, which include the
additional assumption that $C$ and $C^{\prime}$\ are compatible (i.e., they do
not involve the same transition at any vertex). Bouchet's version requires
also that $C$ and $C^{\prime}$ respect the same edge-directions.

Greater generality is always desirable, of course, but it is important to
observe that in fact, Theorem \ref{intinv} is particularly valuable when $C$
and $C^{\prime}$ are compatible. The equation $\mathcal{I}_{C^{\prime}%
}(C)^{-1}=\mathcal{I}_{C}(C^{\prime})$ does not allow us to construct the full
interlacement matrix $\mathcal{I}(C^{\prime})$ directly from $\mathcal{I}(C)$
if $C$ and $C^{\prime}$ share a transition at any vertex, because there is no
way to recover the information that is lost when off-diagonal entries are set
to $0$ in part (i) of Definition \ref{relint}. (For instance, if $C$ is any
Euler system and $v$ is any vertex then Definition \ref{relint} tells us that
$\mathcal{I}_{C}(C)=\mathcal{I}_{C}(C\ast v)=\mathcal{I}_{C\ast v}(C)$ is the
identity matrix no matter how $C$ is structured.) This observation motivates
our last definition.

\begin{definition}
Let $F$ be a 4-regular multigraph with an Euler system $C$, and suppose
$W\subseteq V(F)$ has the property that an Euler system is obtained from $C$
by using the $\psi$ transition at every vertex in $W$, and the $\chi$
transition at every vertex not in $W$. Then this Euler system is the $\iota
$\emph{-transform} of $C$ with respect to $W$; we denote it $C\#W$.
\end{definition}

According to the circuit-nullity formula, $C\#W$ is an Euler system if and
only if a nonsingular matrix $M$ is obtained from $\mathcal{I}(C)$ by changing
the diagonal entries corresponding to elements of $W$ from $0$ to $1$. Then
$M=\mathcal{I}_{C\#W}(C)$ and $M^{-1}=\mathcal{I}_{C}(C\#W)$, so
$\mathcal{I}(C)$ and $\mathcal{I}(C\#W)$ are modified inverses. We call the
process of obtaining $C\#W$ from $C$\ an $\iota$-transformation in recognition
of this connection with matrix inversion.

Theorem \ref{lcinv} and the fact that simple local complementations correspond
to $\kappa$-transformations imply the following alternative to Kotzig's theorem:

\begin{theorem}
\label{tau} All the Euler systems of a 4-regular multigraph can be obtained
from any one by applying finite sequences of $\iota$-transformations.
\end{theorem}

%

\begin{figure}
[ptb]
\begin{center}
\includegraphics[
trim=0.858409in 7.166924in 0.862532in 0.863240in,
height=1.8957in,
width=4.5956in
]%
{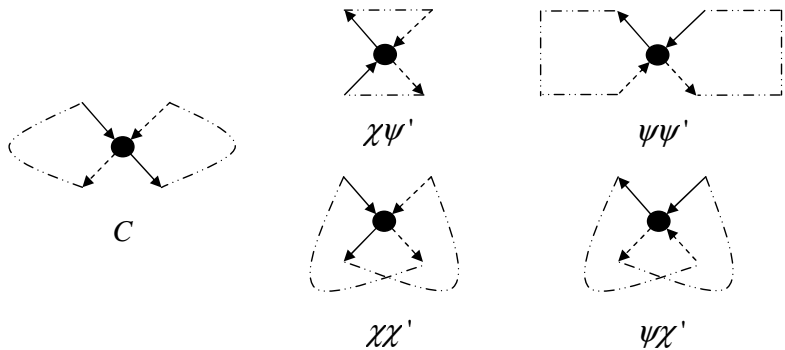}%
\caption{An Euler circuit of $C$ incident at a vertex $v$, along with the four
different ways\ $C\#W$ might be configured at $v$.}%
\label{intext9}%
\end{center}
\end{figure}

The fact that the full interlacement matrix $\mathcal{I}(C\#W)$ is determined
by $\mathcal{I}(C)$ and $W$ implies that the transition labels of $C\#W$ are
also determined. Suppose $v\in V(F)$; we use $\phi,\chi,\psi$ to label the
three transitions at $v$ with respect to $C$, and $\phi^{\prime},\chi^{\prime
},\psi^{\prime}$ to label the three transitions at $v$ with respect to $C\#W$.
Let $M$ denote the nonsingular matrix obtained from $\mathcal{I}(C)$ by
toggling the diagonal entries corresponding to elements of $W$, and let
$X\subseteq V(F)$ denote the set of vertices corresponding to nonzero diagonal
entries of $M^{-1}$. Then $\phi^{\prime}=\psi$ if $v\in W$, $\phi^{\prime
}=\chi$ if $v\not \in W$, $\phi=\psi^{\prime}$ if $v\in X$ and $\phi
=\chi^{\prime}$ if $v\notin X$. Consequently $v\in W\cap X$ implies
$\phi^{\prime}=\psi$, $\phi=\psi^{\prime}$ and $\chi=\chi^{\prime}$; $v\in
W-X$ implies $\phi^{\prime}=\psi$, $\phi=\chi^{\prime}$ and $\chi=\psi
^{\prime}$; $v\in X-W$ implies $\phi^{\prime}=\chi$, $\phi=\psi^{\prime}$ and
$\psi=\chi^{\prime}$; and $v\in V(F)-W-X$ implies $\phi^{\prime}=\chi$,
$\phi=\chi^{\prime}$ and $\psi=\psi^{\prime}$. See\ Figure \ref{intext9},
where the four cases are indexed by first listing $\phi^{\prime}$ with respect
to $C$ and then listing $\phi$ with respect to $C^{\prime}$, so that $\chi
\psi^{\prime}$ represents $v\in X-W$, $\psi\psi^{\prime}$ represents $v\in
W\cap X$, $\psi\chi^{\prime}$ represents $v\in W-X$ and $\chi\chi^{\prime}$
represents $v\in V(F)-W-X$.

Theorem \ref{intinv} also implies that interlacement matrices satisfy a
limited kind of multiplicative functoriality, which we have not seen mentioned elsewhere.

\begin{corollary}
\label{product} Let $C$ and $C^{\prime}$ be Euler systems of a 4-regular
multigraph $F$, and let $P$ be the circuit partition described by: (a) at
every vertex where $C$ and $C^{\prime}$ involve the same transition, $P$
involves the same transition; and (b) at every vertex where $C$ and
$C^{\prime}$ involve two different transitions, $P$ involves the third
transition. Then%
\[
\mathcal{I}_{P}(C)=\mathcal{I}_{C^{\prime}}(C)\cdot\mathcal{I}_{P}(C^{\prime
}).
\]

\end{corollary}

\begin{proof}
Let $I$ be the identity matrix, and let $I^{\prime}$ be the diagonal matrix
whose $vv$ entry is $1$ if and only if $C$ and $C^{\prime}$ involve different
transitions at $v$. Then%
\begin{gather*}
\mathcal{I}_{P}(C)=I^{\prime}+\mathcal{I}_{C^{\prime}}(C)=I+\mathcal{I}%
_{C^{\prime}}(C)\cdot I^{\prime}\\
=\mathcal{I}_{C^{\prime}}(C)\cdot(\mathcal{I}_{C}(C^{\prime})+I^{\prime
})=\mathcal{I}_{C^{\prime}}(C)\cdot\mathcal{I}_{P}(C^{\prime}).
\end{gather*}

\end{proof}

\begin{corollary}
Suppose a 4-regular multigraph $F$ has three pairwise compatible Euler systems
$C$, $C^{\prime}$ and $C^{\prime\prime}$. (That is, no two of $C,C^{\prime
},C^{\prime\prime}$ involve the same transition at any vertex). Then%
\[
\mathcal{I}_{C}(C^{\prime})\cdot\mathcal{I}_{C^{\prime\prime}}(C)\cdot
\mathcal{I}_{C^{\prime}}(C^{\prime\prime})
\]
is the identity matrix.
\end{corollary}

\begin{proof}
By Corollary \ref{product}, $\mathcal{I}_{C^{\prime\prime}}(C)=\mathcal{I}%
_{C^{\prime}}(C)\cdot\mathcal{I}_{C^{\prime\prime}}(C^{\prime})$.
\end{proof}

\section{Examples}%

\begin{figure}
[h]
\begin{center}
\includegraphics[
trim=1.002714in 5.876877in 1.148668in 1.006578in,
height=2.6982in,
width=4.2955in
]%
{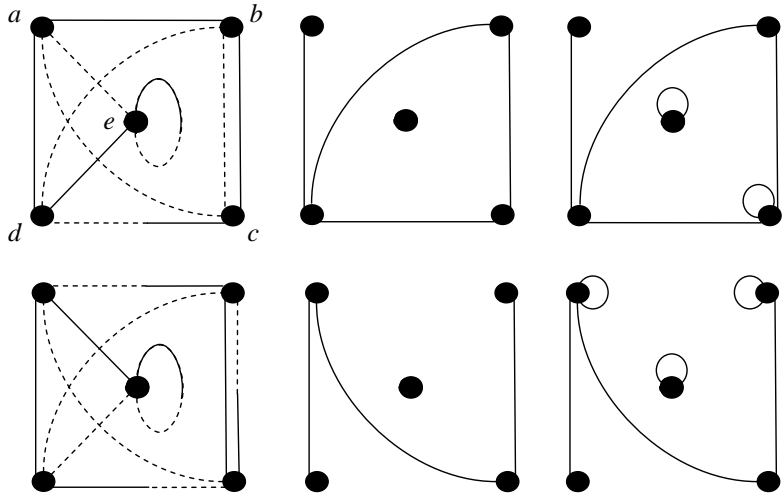}%
\caption{Two\ Euler circuits in the graph of Figure \ref{intext5}, the
interlacement graphs $\mathcal{I}(C)$ and $\mathcal{I}(C^{\prime})$, and the
looped graphs with adjacency matrices $\mathcal{I}_{C^{\prime}}(C)$ and
$\mathcal{I}_{C}(C^{\prime})$.}%
\label{intext5a}%
\end{center}
\end{figure}
Before giving proofs we discuss some illustrative examples. In the first
column of Figure \ref{intext5a} we see two\ Euler circuits $C$ and $C^{\prime
}$ in the graph $F$ of Figure \ref{intext5}. $C$ is given by the double
occurrence word $abcdbcaeed$, and $C^{\prime}$ is given by $abcbdeeadc$. In
the second column we see the simple graphs $\mathcal{I}(C)$ and $\mathcal{I}%
(C^{\prime})$, and in the third column we see the looped graphs whose
adjacency matrices are%
\[
\text{ }\mathcal{I}_{C^{\prime}}(C)=%
\begin{pmatrix}
0 & 0 & 0 & 1 & 0\\
0 & 0 & 1 & 1 & 0\\
0 & 1 & 1 & 1 & 0\\
1 & 1 & 1 & 0 & 0\\
0 & 0 & 0 & 0 & 1
\end{pmatrix}
\text{ and }\mathcal{I}_{C}(C^{\prime})=%
\begin{pmatrix}
1 & 0 & 1 & 1 & 0\\
0 & 1 & 1 & 0 & 0\\
1 & 1 & 0 & 0 & 0\\
1 & 0 & 0 & 0 & 0\\
0 & 0 & 0 & 0 & 1
\end{pmatrix}
.
\]
(The rows and columns correspond to $a,b,c,d,e$ respectively.) Observe that
$\mathcal{I}_{C}(C^{\prime})=\mathcal{I}_{C^{\prime}}(C)^{-1}$, in accordance
with Theorem \ref{intinv}.%
\begin{figure}
[pth]
\begin{center}
\includegraphics[
trim=1.000240in 7.881476in 2.585121in 1.008717in,
height=1.292in,
width=3.2854in
]%
{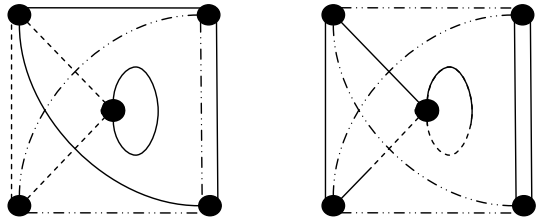}%
\caption{A four-circuit partition $P$ and a three-circuit partition
$P^{\prime}$.}%
\label{intext5b}%
\end{center}
\end{figure}

In\ Figure \ref{intext5b}, we see two circuit partitions in $F$; $P$
corresponds to the set of four words $\{e,ade,abc,bcd\}$ and $P^{\prime}$
corresponds to $\{aeed,bc,abcd\}$. (In order to trace a circuit in the figure,
maintain the dash pattern while traversing each vertex. Note however that the
dash pattern may change in the middle of an edge; this is done to prevent the
same dash pattern from appearing four times at any vertex, as that would
confuse the transitions.) $P$ and $P^{\prime}$ have%
\[
\mathcal{I}_{P}(C)=%
\begin{pmatrix}
0 & 0 & 0 & 0 & 0\\
0 & 1 & 0 & 0 & 0\\
0 & 0 & 0 & 0 & 0\\
0 & 0 & 0 & 1 & 0\\
0 & 0 & 0 & 0 & 0
\end{pmatrix}
\text{, }\mathcal{I}_{P}(C^{\prime})=%
\begin{pmatrix}
1 & 0 & 0 & 0 & 0\\
0 & 1 & 1 & 0 & 0\\
0 & 1 & 1 & 0 & 0\\
0 & 0 & 0 & 0 & 0\\
0 & 0 & 0 & 0 & 0
\end{pmatrix}
,
\]%
\[
\mathcal{I}_{P^{\prime}}(C)=%
\begin{pmatrix}
0 & 0 & 0 & 0 & 0\\
0 & 1 & 1 & 0 & 0\\
0 & 1 & 1 & 0 & 0\\
0 & 0 & 0 & 1 & 0\\
0 & 0 & 0 & 0 & 1
\end{pmatrix}
\text{ and }\mathcal{I}_{P^{\prime}}(C^{\prime})=%
\begin{pmatrix}
1 & 0 & 0 & 0 & 0\\
0 & 0 & 0 & 0 & 0\\
0 & 0 & 1 & 0 & 0\\
0 & 0 & 0 & 0 & 0\\
0 & 0 & 0 & 0 & 1
\end{pmatrix}
.
\]
In accordance with the circuit-nullity formula, $\nu(\mathcal{I}_{P}%
(C))=\nu(\mathcal{I}_{P}(C^{\prime}))=3$ and $\nu(\mathcal{I}_{P^{\prime}%
}(C))=\nu(\mathcal{I}_{P^{\prime}}(C^{\prime}))=2$. As predicted by Theorem
\ref{nullspaces}, the nullspace of $\mathcal{I}_{P}(C)$ is spanned by the
relative core vectors $(0,0,0,0,1)$, $(1,0,0,0,1)$, $(1,0,1,0,0)$ and
$(0,0,1,0,0)$. The nullspace of $\mathcal{I}_{P}(C^{\prime})$ is spanned by
the relative core vectors $(0,0,0,0,1)$, $(0,0,0,1,1)$, $(0,1,1,0,0)$ and
$(0,1,1,1,0)$. The nullspace of $\mathcal{I}_{P^{\prime}}(C)$ is spanned by
$(1,0,0,0,0)$, $(0,1,1,0,0)$ and $(1,1,1,0,0)$; and the nullspace of
$\mathcal{I}_{P^{\prime}}(C^{\prime})$ is spanned by $(0,0,0,1,0)$,
$(0,1,0,0,0)$ and $(0,1,0,1,0)$.

\section{Theorem \ref{lcinv}}

\subsection{Modified inverses from local complementation}

Theorem \ref{lcinv} concerns the equivalence relation $\sim_{lc}$ on
zero-diagonal symmetric matrices generated by simple local complementations
$S\mapsto S^{i}$, where $i$ is an arbitrary index. As noted in the
introduction, Theorem \ref{lcinv} is related to results of Brijder and
Hoogeboom \cite{BH, BH1, BH2}, Glantz and Pelillo \cite{GP} and Ilyutko
\cite{I1, I} regarding an equivalence relation $\sim_{piv}$ on symmetric
matrices that is generated by two kinds of operations: (non-simple) local
complementations $S\mapsto S_{ns}^{i}$ where the $i^{th}$ diagonal entry is
$1$, and edge pivots $S\mapsto S^{ij}$ where the $i^{th}$ and $j^{th}$
diagonal entries are $0$ and the $ij$ and $ji$ entries are $1$. (The reader
familiar with the work of Arratia, Bollob\'{a}s and Sorkin \cite{A1, A2, A}
should be advised that the operation they denote $S^{ij}$ differs from this
one by interchanging the $i^{th}$ and $j^{th}$ rows and columns.) It is well
known that an edge pivot $S^{ij}$ of a symmetric matrix $S$ coincides with two
triple simple local complementations, $S^{ij}=((S^{i})^{j})^{i}=((S^{j}%
)^{i})^{j}$. (See \cite{BH1} for an extension of this familiar equality to set
systems.) This description of the pivot is not very useful in connection with
$\sim_{piv}$, as the individual local complementations with respect to $i$ and
$j$ are \textquotedblleft illegal\textquotedblright\ for $\sim_{piv}$ when the
$i^{th}$ and $j^{th}$ diagonal entries are $0$. It is useful for us, though:
as non-simple local complementation coincides with simple local
complementation off the diagonal, the description of the pivot implies that if
$S$ and $T$ are zero-diagonal symmetric matrices, and $S^{\prime}$ and
$T^{\prime}$ are symmetric matrices that are (respectively) equal to $S$ and
$T$ except for diagonal entries, then $S^{\prime}\sim_{piv}T^{\prime
}\Rightarrow S\sim_{lc}T$.

This implication allows us to prove half of Theorem \ref{lcinv} using a result
noted in \cite{BH, BH1, BH2, GP}, namely that symmetric matrices related by
the principal pivot transform are also related by $\sim_{piv}$. Suppose
$S_{1}$ is a zero-diagonal symmetric matrix and $S_{2}$ is a modified inverse
of $S_{1}$. Then there are nonsingular symmetric matrices $S_{1}^{\prime}$ and
$S_{2}^{\prime}$ that are (respectively) equal to $S_{1}$ and $S_{2}$ except
for diagonal entries, and have $(S_{1}^{\prime})^{-1}=S_{2}^{\prime}$. As
inverses, $S_{1}^{\prime}$ and $S_{2}^{\prime}$ are related by the principal
pivot transform, so $S_{1}^{\prime}\sim_{piv}S_{2}^{\prime}$; consequently the
implication of the last paragraph tells us that $S_{1}\sim_{lc}S_{2}$.
Iterating this argument, we conclude that $S\sim_{mi}T\Rightarrow S\sim_{lc}T$.

\subsection{A brief discussion of the principal pivot transform}

The principal pivot transform (ppt) was introduced by Tucker \cite{Tu}; a
survey of its properties was given by Tsatsomeros \cite{Ts}. The reader who
would like to learn about the ppt and its relationship with graph theory
should consult these papers,\ the work of Brijder and Hoogeboom \cite{BH, BH1,
BH2} and Glantz and Pelillo \cite{GP}, and the references given there. For the
convenience of the reader who simply wants to understand the argument above,
we sketch the details briefly.

The principal pivot transform is valuable for us because it provides a way to
obtain the inverse of a nonsingular symmetric $GF(2)$-matrix incrementally,
using pivots and non-simple local complementations. This property is not
immediately apparent from the definition:

\begin{definition}
\label{ppt} Suppose
\[
M=%
\begin{pmatrix}
P & Q\\
R & S
\end{pmatrix}
\]
is an $n\times n$ matrix with entries in a field $F$, and suppose $P$ is a
nonsingular principal submatrix involving the rows and columns whose indices
lie in a set $X$. Then
\[
M\ast X=%
\begin{pmatrix}
P^{-1} & -P^{-1}Q\\
RP^{-1} & S-RP^{-1}Q
\end{pmatrix}
.
\]
In particular, if $M$ is nonsingular then $M\ast\{1,...,n\}=M^{-1}$.
\end{definition}

The matrix $M$ is displayed in the given form only for convenience; the
definition may be applied to any set of indices $X$ whose corresponding
principal submatrix $P$ is nonsingular.

A direct calculation (which we leave to the reader) yields an alternative
characterization: if we think of $F^{n}$ as the direct sum of a subspace
corresponding to indices from $X$ and a subspace corresponding to the rest of
the indices, then the linear endomorphism of $F^{n}$ corresponding to $M\ast
X$ is related to the linear endomorphism corresponding to $M$ by the relation
\[
M%
\begin{pmatrix}
x_{1}\\
x_{2}%
\end{pmatrix}
=%
\begin{pmatrix}
y_{1}\\
y_{2}%
\end{pmatrix}
\text{ if and only if }(M\ast X)%
\begin{pmatrix}
y_{1}\\
x_{2}%
\end{pmatrix}
=%
\begin{pmatrix}
x_{1}\\
y_{2}%
\end{pmatrix}
.
\]
This characterization directly implies two useful properties.

(1) Suppose $X_{1}$ and $X_{2}$ are subsets of $\{1,...,n\}$, and $(M\ast
X_{1})\ast X_{2}$ is defined. Then $M\ast(X_{1}\Delta X_{2})$ is also defined,
and it is the same as $(M\ast X_{1})\ast X_{2}$. (Here $\Delta$ denotes the
symmetric difference, $X_{1}\Delta X_{2}=(X_{1}\cup X_{2})-(X_{1}\cap X_{2})$.)

(2) If $M$ is nonsingular then the submatrix \thinspace$S-RP^{-1}Q$ of $M\ast
X$ must be nonsingular too. For%
\[
(M\ast X)%
\begin{pmatrix}
\mathbf{0}\\
x_{2}%
\end{pmatrix}
=%
\begin{pmatrix}
x_{1}\\
\mathbf{0}%
\end{pmatrix}
\text{ implies }M%
\begin{pmatrix}
x_{1}\\
x_{2}%
\end{pmatrix}
=%
\begin{pmatrix}
\mathbf{0}\\
\mathbf{0}%
\end{pmatrix}
,
\]
which requires $x_{1}=\mathbf{0}$ and $x_{2}=\mathbf{0}$.

Note also that if $M$ is a symmetric $GF(2)$-matrix, then so is $M\ast X$.

Suppose now that $M$ is a nonsingular symmetric $GF(2)$-matrix. We begin a
recursive calculation by performing principal pivot transforms $\ast
\{i_{1}\},\ast\{i_{2}\}$, ..., $\ast\{i_{k}\}$ for as long as we can, subject
to the proviso that $i_{1}$, ..., $i_{k}$ are pairwise distinct. It is a
direct consequence of Definition \ref{ppt} that each transformation
$\ast\{i_{j}\}$ is the same as non-simple local complementation with respect
to $i_{j}$, because $P^{-1}$ is the $1\times1$ matrix $(1)$, $R$ is the
transpose of $Q$, and an entry of $RQ$ is $1$ if and only if its row and
column correspond to nonzero entries of $R$ and $Q$. For convenience we
display the special case $\{i_{1},...,i_{k}\}=\{1,...,k\}$:%
\[
M\ast\{i_{1},...,i_{k}\}=M\ast\{i_{1}\}\ast\cdots\ast\{i_{k}\}=M^{\prime}=%
\begin{pmatrix}
P^{\prime} & Q^{\prime}\\
R^{\prime} & S^{\prime}%
\end{pmatrix}
,
\]
where the diagonal entries of $P^{\prime}$ (resp. $S^{\prime}$) are all $1$
(resp. all $0$) and $R^{\prime}$ is the transpose of $Q^{\prime}.$ Note that
by property (2), $S^{\prime}$ is nonsingular; hence each row of $S^{\prime}$
must certainly contain a nonzero entry. Consequently we may (again for
convenience) display the matrix in a different way:%
\[
M\ast\{i_{1},...,i_{k}\}=M^{\prime}=%
\begin{pmatrix}
P^{\prime\prime} & Q^{\prime\prime}\\
R^{\prime\prime} & S^{\prime\prime}%
\end{pmatrix}
\text{ where }P^{\prime\prime}=%
\begin{pmatrix}
0 & 1\\
1 & 0
\end{pmatrix}
.
\]
The next step in the calculation is to perform the principal pivot transform
$\ast\{i_{k+1},i_{k+2}\}$, where $i_{k+1}$ and $i_{k+2}$ are the indices
involved in $P^{\prime\prime}$. A direct calculation using Definition
\ref{ppt} shows that
\[
M^{\prime}\ast\{i_{k+1},i_{k+2}\}=M^{\prime\prime}=%
\begin{pmatrix}
P^{\prime\prime} & Q^{\prime\prime\prime}\\
R^{\prime\prime\prime} & S^{\prime\prime}+X
\end{pmatrix}
\]
where $Q^{\prime\prime\prime}$ is obtained by interchanging the two rows of
$Q^{\prime\prime}$, $R^{\prime\prime\prime}$ is the transpose of
$Q^{\prime\prime\prime}$, and $X$ is the matrix whose $ij$ entry is
\[
(R^{\prime\prime})_{ii_{k+1}}(Q^{\prime\prime})_{i_{k+2}j}+(R^{\prime\prime
})_{ii_{k+2}}(Q^{\prime\prime})_{i_{k+1}j}=(Q^{\prime\prime})_{i_{k+1}%
i}(Q^{\prime\prime})_{i_{k+2}j}+(Q^{\prime\prime})_{i_{k+2}i}(Q^{\prime\prime
})_{i_{k+1}j}.
\]
That is, $X_{ij}=0$ unless the $i^{th}$ and $j^{th}$ columns of $Q^{\prime
\prime}$ are distinct nonzero vectors. It follows that $M^{\prime\prime}$ is
the pivot $(M^{\prime})^{i_{k+1}i_{k+2}}$; equivalently, $M^{\prime\prime}$ is
the triple simple local complement $(((M^{\prime})^{i_{k+1}})^{i_{k+2}%
})^{i_{k+1}}$.

The resulting matrix $M^{\prime\prime}$ has no new nonzero diagonal entry, so
the second step may be repeated as many times as necessary. At no point do we
re-use an index that has already been involved in a principal pivot transform,
at no point do we obtain a non-symmetric matrix, and at no point is the
principal submatrix determined by the as-yet unused indices singular.
Consequently the calculation proceeds until every index has been involved in
precisely one principal pivot transform. By property (1), at the end of the
calculation we have obtained $M\ast\{1,...,n\}=M^{-1}$ using individual steps
each of which is either a non-simple local complementation or a pivot.

\subsection{The $\nu,\nu,\nu+1$ lemma}

Lemma 2 of Balister, Bollob\'{a}s, Cutler and Pebody \cite{BBCS} is a very
useful result about the nullities of three related matrices, which we cite in
the next section. Here is a sharpened form of the lemma, involving the
nullspaces of the three matrices rather than only their nullities.

\begin{lemma}
\label{twothree} Suppose $M$ is a symmetric $GF(2)$-matrix. Let $\rho$ be an
arbitrary row vector, and let $\mathbf{0}$ be the row vector with all entries
$0$; denote their transposes $\kappa$ and $\mathbf{0}$ respectively. Let
$M_{1},M_{2}$ and $M_{3}$ denote the indicated symmetric matrices.%
\[
M_{1}=%
\begin{pmatrix}
M & \kappa\\
\rho & 1
\end{pmatrix}
\quad\quad M_{2}=%
\begin{pmatrix}
M & \kappa\\
\rho & 0
\end{pmatrix}
\quad\quad M_{3}=%
\begin{pmatrix}
M & \mathbf{0}\\
\mathbf{0} & 1
\end{pmatrix}
\]
Then two of $M_{1},M_{2},M_{3}$ have the same nullspace, say of dimension
$\nu$. The nullspace of the remaining matrix has dimension $\nu+1$, and it
contains the nullspace shared by the other two.
\end{lemma}

\begin{proof}
We use $rk(N)$ and $\nu(N)$ to denote the rank and nullity of a matrix $N$,
$\operatorname{row}(N)$ and $\operatorname{col}(N)$ to denote the spaces
spanned by the rows and columns of $N$, and $\ker N$ to denote the nullspace
of $N$, i.e., the space of row vectors $x$ with $x\cdot N=\mathbf{0}$.

\textbf{Case 1}. Suppose $\rho\notin\operatorname{row}(M)$; then also
$\kappa\notin\operatorname{col}(M)$.

Observe that $\nu(M_{1})=\nu(M_{2})=\nu(M_{3})-1=\nu(M)-1$ in this case.

We claim that $\ker M_{3}\supseteq(\ker M_{1})\cap(\ker M_{2})$. Note first
that no row vector $%
\begin{pmatrix}
y & 1
\end{pmatrix}
$ can be an element of both $\ker M_{1}$ and $\ker M_{2}$, because $%
\begin{pmatrix}
y & 1
\end{pmatrix}
\cdot M_{1}=\mathbf{0}\Rightarrow y\cdot\kappa=1$ and $%
\begin{pmatrix}
y & 1
\end{pmatrix}
\cdot M_{2}=\mathbf{0}\Rightarrow y\cdot\kappa=0$. Consequently every
$x\in(\ker M_{1})\cap(\ker M_{2})$ is of the form $%
\begin{pmatrix}
y & 0
\end{pmatrix}
$ for some $y\in\ker M$. Obviously
\[
\ker M_{3}=\left\{
\begin{pmatrix}
y & 0
\end{pmatrix}
\text{
$\vert$
}y\in\ker M\right\}  ,
\]
so the claim is verified.

The inclusion
\[
(\ker M_{1})\cap(\ker M_{2})\supseteq\left\{
\begin{pmatrix}
y & 0
\end{pmatrix}
\text{
$\vert$
}y\in\ker M\text{ and }y\cdot\kappa=0\right\}
\]
is obvious, but note that the dimension of the right hand side is at least
$\nu(M)-1$. As $\nu(M)-1=\nu(M_{1})=\nu(M_{2})$, we conclude that
\[
\ker M_{1}=\ker M_{2}=(\ker M_{1})\cap(\ker M_{2})=\left\{
\begin{pmatrix}
y & 0
\end{pmatrix}
\text{
$\vert$
}y\in\ker M\text{ and }y\cdot\kappa=0\right\}  .
\]

\textbf{Case 2}. Suppose $\rho\in\operatorname{row}(M)$; then also $\kappa
\in\operatorname{col}(M)$.

As $\rho\in\operatorname{row}(M)$, at least one of $%
\begin{pmatrix}
\rho & 0
\end{pmatrix}
,%
\begin{pmatrix}
\rho & 1
\end{pmatrix}
$ is an element of $\operatorname{row}%
\begin{pmatrix}
M & \kappa
\end{pmatrix}
$. If both are elements then
\[
rk%
\begin{pmatrix}
M & \kappa
\end{pmatrix}
=rk%
\begin{pmatrix}
M & \kappa\\
\rho & 0\\
\rho & 1
\end{pmatrix}
=rk%
\begin{pmatrix}
M & \kappa\\
\rho & 0\\
\mathbf{0} & 1
\end{pmatrix}
=rk%
\begin{pmatrix}
M & \mathbf{0}\\
\rho & 0\\
\mathbf{0} & 1
\end{pmatrix}
=rk(M)+1
\]
and hence $\kappa\notin\operatorname{col}(M)$, an impossibility. Consequently
one of $M_{1},M_{2}$ is of rank $rk(M)$, and the other is of rank
$rk(M)+1=rk(M_{3})$.

We claim that $\ker M_{3}\subseteq(\ker M_{1})\cap(\ker M_{2})$. Suppose
$x\in\ker M_{3}$; clearly then $x=%
\begin{pmatrix}
y & 0
\end{pmatrix}
$ for some $y\in\ker M$. As $y\in\ker M$, $y\cdot z=0$ $\forall z\in
\operatorname{col}(M)$. It follows that $y\cdot\kappa=0$, and hence $x\in(\ker
M_{1})\cap(\ker M_{2})$ as claimed.

As one of $M_{1},M_{2}$ has nullity $\nu(M_{3})$ and the other has nullity
$\nu(M_{3})+1$, the claim implies that either $\ker M_{3}=\ker M_{2}%
\subset\ker M_{1}$ or $\ker M_{3}=\ker M_{1}\subset\ker M_{2}$.
\end{proof}

\subsection{Local complements from modified inversion}

The implication $S\sim_{lc}T\Rightarrow S\sim_{mi}T$ does not follow directly
from results regarding $\sim_{piv}$, but our argument uses two lemmas very
similar to ones used by Ilyutko \cite{I1}.

\begin{lemma}
\label{two}Let $S$ be an $n\times n$ zero-diagonal symmetric $GF(2)$-matrix.
Suppose $1\leq i\leq n$ and the $i^{th}$ row of $S$ has at least one nonzero
entry. Then there is a nonsingular symmetric $GF(2)$-matrix $M$ that differs
from $S$ only in diagonal entries other than the $i^{th}$.
\end{lemma}

\begin{proof}
We may as well presume that $i=1$, and that $s_{12}\neq0$.

For $k\geq2$ let $S_{k}$ be the submatrix of $S$ obtained by deleting all rows
and columns with indices $>k$. We claim that for every $k\geq2$, there is a
subset $T_{k}\subseteq\{3,...,k\}$ with the property that we obtain a
nonsingular matrix $S_{k}^{\prime}$ by toggling the diagonal entries of
$S_{k}$ with indices in $T_{k}$. When $k=2$ the claim is satisfied by
$T_{2}=\emptyset.$

Proceeding inductively, suppose $k\geq2$ and $T_{k}\subseteq\{3,...,k\}$
satisfies the claim. If we obtain a nonsingular matrix by toggling the
diagonal entries of $S_{k+1}$ with indices in $T_{k}$, then $T_{k+1}=T_{k}$
satisfies the claim for $k+1$. If not, then the $\nu,\nu,\nu+1$ lemma implies
that $T_{k+1}=T_{k}\cup\{k+1\}$ satisfies the claim.

The required matrix $M$ is obtained from $S=S_{n}$ by toggling the diagonal
entries with indices in $T_{n}$.
\end{proof}

\begin{lemma}
\label{one}Suppose
\[
M=%
\begin{pmatrix}
0 & \mathbf{1} & \mathbf{0}\\
\mathbf{1} & M_{11} & M_{12}\\
\mathbf{0} & M_{12} & M_{22}%
\end{pmatrix}
\]
is a nonsingular symmetric matrix, with
\[
M^{-1}=%
\begin{pmatrix}
a & \rho\\
\kappa & N
\end{pmatrix}
.
\]
Then the matrix
\[
M^{\prime}=%
\begin{pmatrix}
0 & \mathbf{1} & \mathbf{0}\\
\mathbf{1} & \bar{M}_{11} & M_{12}\\
\mathbf{0} & M_{12} & M_{22}%
\end{pmatrix}
\]
obtained by toggling all entries within the block $M_{11}$ is also
nonsingular, and
\[
(M^{\prime})^{-1}=%
\begin{pmatrix}
a+1 & \rho\\
\kappa & N
\end{pmatrix}
.
\]

\end{lemma}

\begin{proof}
$M^{-1}\cdot M$ is the identity matrix, so%
\[%
\begin{pmatrix}
a & \rho
\end{pmatrix}
\cdot%
\begin{pmatrix}
0\\
\mathbf{1}\\
\mathbf{0}%
\end{pmatrix}
=1\text{;}%
\]
consequently the row vector $\rho$ must have an odd number of nonzero entries
in the columns corresponding to the $\mathbf{1}$ in the first row of $M$.
Similarly, each row of $N$ must have an even number of nonzero entries in
these columns, because
\[%
\begin{pmatrix}
\kappa & N
\end{pmatrix}
\cdot%
\begin{pmatrix}
0\\
\mathbf{1}\\
\mathbf{0}%
\end{pmatrix}
=\mathbf{0}.
\]
It follows that the products%
\[%
\begin{pmatrix}
a & \rho\\
\kappa & N
\end{pmatrix}
\cdot M\text{ and }%
\begin{pmatrix}
a+1 & \rho\\
\kappa & N
\end{pmatrix}
\cdot M^{\prime}%
\]
are equal.
\end{proof}

As stated, Lemma \ref{one} requires that the first row of $M$ be in the form $%
\begin{pmatrix}
0 & \mathbf{1} & \mathbf{0}%
\end{pmatrix}
$. This is done only for convenience; the general version of the lemma applies
to any row in which the diagonal entry is $0$.

Suppose now that $S$ is a zero-diagonal symmetric $n\times n$ matrix and
$1\leq i\leq n$. If every entry of the $i^{th}$ row of $S$ is $0$, then the
local complement $S^{i}$ is the same as $S$.

If the $i^{th}$ row of $S$ includes some nonzero entry, then Lemma \ref{two}
tells us that there is a nonsingular matrix $M$ such that (a) $M$ is obtained
from $S$ by toggling some diagonal entries other than the $i^{th}$. The
general version of Lemma \ref{one} then tells us that there is a nonsingular
matrix $M^{\prime}$ such that (b) $(M^{\prime})^{-1}$ equals $M^{-1}$ except
for the $i^{th}$ diagonal entry and (c) $M^{\prime}$ equals the local
complement $S^{i}$ except for diagonal entries. Condition (a) implies that a
modified inverse $T$ of $S$ is obtained by changing all diagonal entries of
$M^{-1}$ to $0$; condition (b) implies that $T$ is also equal to $(M^{\prime
})^{-1}$ except for diagonal entries; and condition (c) implies that the local
complement $S^{i}$ is a modified inverse of $T$.

It follows that every simple local complement $S^{i}$ can be obtained from $S$
using no more than two modified inversions. Applying this repeatedly, we
conclude that $S\sim_{lc}T\Rightarrow S\sim_{mi}T$.

\section{Theorem \ref{nullspaces}}

Let $F$ be a 4-regular multigraph with an Euler system $C$, and suppose $v\in
V(F)$. Let the Euler circuit of $C$ incident at $v$ be $vC_{1}vC_{2}v$, where
$v$ does not appear within $C_{1}$ or $C_{2}$. Every edge of the connected
component of $F$ that contains $v$ lies on precisely one of $C_{1},C_{2}$. If
$w\neq v$ is a vertex of this component which is not a neighbor of $v$ in
$\mathcal{I}(C)$, i.e., $v$ and $w$ are not interlaced with respect to $C$,
then all four half-edges incident at $w$ appear on the same $C_{i}$. On the
other hand, if $v$ and $w$ are neighbors in $\mathcal{I}(C)$ then two of the
four half-edges incident at $w$ appear on $C_{1}$, and the other two appear on
$C_{2}$. Moreover the only transition at $w$ that pairs together the
half-edges from the same $C_{i}$ is the $\phi$ transition; the $\chi$ and
$\psi$ transitions at $w$ pair each half-edge from $C_{1}$ with a half-edge
from~$C_{2}$. At $v$, instead, only the $\chi$ transition pairs together
half-edges from the same $C_{i}$; the $\phi$ and $\psi$ transitions pair each
half-edge from $C_{1}$ with a half-edge from $C_{2}$. These observations are
summarized in the table below, where $N(v)$ denotes the set of neighbors of
$v$ in $\mathcal{I}(C)\,$, $(1)(2)$ indicates a transition that pairs together
half-edges from the same $C_{i}$ and $(12)$ indicates a transition that pairs
each half-edge from $C_{1}$ with a half-edge from $C_{2}$.

\begin{center}%
\begin{tabular}
[c]{|c|c|c|c|c|}\hline
& vertex & $v$ & $w\in N(v)$ & $w\notin N(v)$\\\hline
transition &  &  &  & \\\hline
$\phi$ &  & $(12)$ & $(1)(2)$ & $(1)(2)$\\\hline
$\chi$ &  & $(1)(2)$ & $(12)$ & $(1)(2)$\\\hline
$\psi$ &  & $(12)$ & $(12)$ & $(1)(2)$\\\hline
\end{tabular}

\end{center}

Suppose $\gamma$ is a circuit that is singly incident at $v$, and involves the
$\phi$ transition at $v$. We start following $\gamma$ at $v$, on a half-edge
that belongs to $C_{i}$. We return to $v$ on a half-edge that belongs to
$C_{j}$, $j\neq i$, so while following $\gamma$ we must have switched between
$C_{i}$ and $C_{j}$ an odd number of times. (See Figure \ref{intext6}.) That
is,
\[
\left\vert \{w\in N(v)\text{
$\vert$
}\gamma\text{ is singly incident at }w\text{ and involves the }\chi\text{ or
}\psi\text{ transition at }w\}\right\vert
\]
must be odd.%
\begin{figure}
[ptb]
\begin{center}
\includegraphics[
trim=1.143720in 8.624910in 1.151142in 0.718832in,
height=0.9729in,
width=4.1943in
]%
{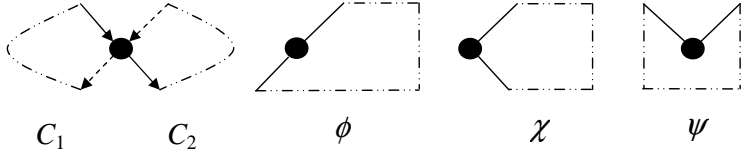}%
\caption{An Euler circuit and three circuits that are singly incident at a
vertex.}%
\label{intext6}%
\end{center}
\end{figure}

Suppose instead that $\gamma$ is a circuit that is singly incident at $v$ and
involves the $\chi$ transition at $v$. If we start following $\gamma$ along a
half-edge that belongs to $C_{i}$, we must return to $v$ on the other
half-edge that belongs to $C_{i}$, so we must have switched between $C_{i}$
and $C_{j}$ an even number of times. Consequently
\[
\left\vert \{w\in N(v)\text{
$\vert$
}\gamma\text{ is singly incident at }w\text{ and involves the }\chi\text{ or
}\psi\text{ transition at }w\}\right\vert
\]
must be even.

Similarly, if $\gamma$ is a circuit that is singly incident at $v$ and
involves the $\psi$ transition at $v$ then
\[
\left\vert \{w\in N(v)\text{
$\vert$
}\gamma\text{ is singly incident at }w\text{ and involves the }\chi\text{ or
}\psi\text{ transition at }w\}\right\vert
\]
must be odd.

Suppose $\gamma$ is doubly incident at $v$ and we follow $\gamma$ after
leaving $v$ on a half-edge that belongs to $C_{1}$. We cannot be sure on which
half-edge we will first return to $v$. But after leaving again, we must return
on the one remaining half-edge. If the first return is from $C_{1}$, then on
the way we must traverse an even number of vertices $w\in N(v)$ such that
$\gamma$ is singly incident at $w$ and involves the $\chi$ or $\psi$
transition at $w$; moreover when we follow the second part of $\gamma$ we will
leave $v$ in $C_{2}$ and also return in $C_{2}$, so again we must encounter an
even number of such vertices. If the first return is from $C_{2}$ instead,
then on the way we must traverse an odd number of vertices $w\in N(v)$ such
that $\gamma$ is singly incident at $w$ and involves the $\chi$ or $\psi$
transition at $w$; the second part of the circuit will begin in $C_{i}$ and
end in $C_{j}$ ($j\neq i$), so it will also include an odd number of such
vertices. In any case,
\[
\left\vert \{w\in N(v)\text{
$\vert$
}\gamma\text{ is singly incident at }w\text{ and involves the }\chi\text{ or
}\psi\text{ transition at }w\}\right\vert
\]
must be even.

\begin{proposition}
\label{ker}Let $F$ be a 4-regular multigraph with an Euler system $C$, and
suppose $P$ is a circuit partition of $F$. Then $\rho(\gamma,C)\in
\ker\mathcal{I}_{P}(C)$ $\forall\gamma\in P$.
\end{proposition}

\begin{proof}
Recall that $\rho(\gamma,C)\in GF(2)^{V(F)}$ has nonzero entries corresponding
to the vertices where $\gamma$ is singly incident and does not involve the
$\phi$ transition; we consider $\rho(\gamma,C)$ as a row vector. Let $v\in
V(F)$, and let $\kappa(v)$ be the column of $\mathcal{I}_{P}(C)$ corresponding
to $v$.

If $\gamma$ lies in a different connected component than $v$, or if $P$
involves the $\phi$ transition at $v$, then for every $w\in V(F)$ at least one
of $\rho(\gamma,C)$ and $\kappa(v)$ has its entry corresponding to $w$ equal
to $0$; consequently $\rho(\gamma,C)\cdot\kappa(v)=0$.

If $\gamma$ lies in the same component as $v$ and $P$ involves the $\chi$
transition at $v$, then considering the definitions of $\rho(\gamma,C)$ and
$\mathcal{I}_{P}(C)$, we see that $\rho(\gamma,C)\cdot\kappa(v)$ is the mod 2
parity of
\[
\left\vert \{w\in N(v)\text{
$\vert$
}\gamma\text{ is singly incident at }w\text{ and involves the }\chi\text{ or
}\psi\text{ transition at }w\}\right\vert .
\]
Whether $\gamma$ is singly or doubly incident at $v$, this number is even as
observed above.

If $\gamma$ lies in the same component as $v$ and $P$ involves the $\psi$
transition at $v$, then considering the definitions of $\rho(\gamma,C)$ and
$\mathcal{I}_{P}(C)$, we see that $\rho(\gamma,C)\cdot\kappa(v)$ is the mod 2
parity of this sum:%
\begin{gather*}
\left\vert \{w\in N(v)\text{
$\vert$
}\gamma\text{ is singly incident at }w\text{ and involves the }\chi\text{ or
}\psi\text{ transition at }w\}\right\vert \\
+\left\{
\begin{tabular}
[c]{ll}%
1, & if $\gamma$ is singly incident at $v$\\
0, & if $\gamma$ is doubly incident at $v$%
\end{tabular}
\ \ \ \ \ \ \right.
\end{gather*}
Whether $\gamma$ is singly or doubly incident at $v$, the sum is even as
observed above.
\end{proof}

\begin{proposition}
\label{onept}Let $F$ be a 4-regular multigraph with an Euler system $C$, and
let $P$ be a circuit partition of $F$. Suppose $Q\subset P$ and there is at
least one connected component of $F$ for which $Q$ contains some but not all
of the incident circuits of $P$. Then there is at least one vertex $v\in V(F)$
such that $P$ involves the $\chi$ or $\psi$ transition at $v$, precisely one
circuit of $Q$ is incident at $v$, and this incident circuit of $Q$ is only
singly incident.
\end{proposition}

\begin{proof}
Let $F_{0}$ be a connected component of $F$ in which $Q$ includes some but not
all of the incident circuits of $P$. Then there must be an edge of $F_{0}$ not
included in any circuit of $Q$.

Choose such an edge, $e_{1}$. If a circuit of $Q$ is incident on an end-vertex
of $e_{1}$, then that circuit is only singly incident. If not, choose an edge
$e_{2}$ that connects an end-vertex of $e_{1}$ to a vertex that is not
incident on $e_{1}$. Continuing this process, we must ultimately find a vertex
at which precisely one circuit of $Q$ is incident, and this circuit is only
singly incident.

Suppose $P$ involves the $\phi$ transition at every such vertex. Then every
circuit of $P$ that is contained in $F_{0}$ and not included in $Q$ involves
only $\phi$ transitions. This is impossible, as the only circuit contained in
$F_{0}$ that involves only $\phi$ transitions is the Euler circuit of $F_{0}$
included in $C$.
\end{proof}

\begin{corollary}
\label{ind}Let $F$ be a 4-regular multigraph with an Euler system $C$, and
suppose $P$ is a circuit partition of $F$. Suppose $Q\subset P$ and there is
no connected component of $F$ for which $Q$ contains every incident circuit of
$P$. Then $\{$relative core vectors of the circuits of $Q\}$ is linearly independent.
\end{corollary}

\begin{proof}
If $Q$ is empty then recall that $\emptyset$ is independent by convention.
Otherwise, let $Q^{\prime}$ be any nonempty subset of $Q$. Proposition
\ref{onept} tells us that there is a vertex $v$ of $F$ at which $P$ involves
the $\chi$ or $\psi$ transition, precisely one circuit of $Q^{\prime}$ is
incident, and this circuit is singly incident. It follows that the $v$
coordinate of
\[
\sum_{\gamma\in Q^{\prime}}\rho(\gamma,C)
\]
is $1$, and hence%
\[
\sum_{\gamma\in Q^{\prime}}\rho(\gamma,C)\not =\mathbf{0}\text{.}%
\]

\end{proof}

Proposition \ref{ker} and Corollary \ref{ind} tell us that the relative core
vectors of the circuits of a circuit partition $P$ span a $(\left\vert
P\right\vert -c(F))$-dimensional subspace of $\ker\mathcal{I}_{P}(C)$. The
circuit-nullity formula tells us that $\left\vert P\right\vert -c(F)$ is the
nullity of $\mathcal{I}_{P}(C)$, so we conclude that the relative core vectors
span the nullspace of $\mathcal{I}_{P}(C)$.

This completes our proof of Theorem \ref{nullspaces}. Before proceeding we
should recall the work of Jaeger \cite{J1}, whose Proposition 4 is equivalent
to the special case of Theorem \ref{nullspaces} in which $P$ involves no
$\phi$ transitions. A different way to prove Theorem \ref{nullspaces} is to
reduce to the special case using \emph{detachment} \cite{F1, N1} along $\phi$
transitions. We prefer the argument above because it avoids the conceptual
complications introduced by Jaeger's use of chord diagrams and surface imbeddings.

\section{Theorem \ref{intinv}}

Suppose $C$ and $C^{\prime}$ are Euler systems of $F$, and $v\in V(F)$. We use
transition labels $\phi,\chi,\psi$ with respect to $C$, and $\phi^{\prime
},\chi^{\prime},\psi^{\prime}$ with respect to $C^{\prime}$.

If $C$ and $C^{\prime}$ involve the same transition at $v$, then the row and
column of both $\mathcal{I}_{C}(C^{\prime})$ and $\mathcal{I}_{C^{\prime}}(C)$
corresponding to $v$ are the same as those of the identity matrix, so the row
and column of $\mathcal{I}_{C}(C^{\prime})\cdot\mathcal{I}_{C^{\prime}}(C)$
corresponding to $v$ are the same as those of the identity matrix.

Suppose instead that $C$ and $C^{\prime}$ involve different transitions at
$v$. Then $C^{\prime}$ involves the $\chi$ or the $\psi$ transition, and $C$
involves the $\chi^{\prime}$ or the $\psi^{\prime}$ transition; the four
possible combinations are illustrated in Figure \ref{intext9} of the
introduction. N.b. The caption of Figure \ref{intext9} mentions $C\#W$, but
the figure is valid for any Euler circuits that do not involve the same
transition at $v$.

Let $vC_{1}vC_{2}v$ be the circuit of $C$ incident at $v$; then $vC_{1}v$ and
$vC_{2}v$ are the two circuits obtained by \textquotedblleft
short-circuiting\textquotedblright\ $C$ at $v$. We claim that the relative
core vector $\rho(vC_{1}v,C^{\prime})$ is the same as the row of
$\mathcal{I}_{C^{\prime}}(C)$ corresponding to $v$. The entry of $\rho
(vC_{1}v,C^{\prime})$ corresponding to a vertex $w\neq v$ is $1$ if and only
if $vC_{1}v$ is singly incident at $w$, and $vC_{1}v$ does not involve the
$\phi^{\prime}$ transition at $w$. Also, the $vw$ entry of $\mathcal{I}%
_{C^{\prime}}(C)$ is $1$ if and only if $v$ and $w$ are interlaced with
respect to $C$, and $C^{\prime}$ does not involve the $\phi$ transition at
$w$. As $vC_{1}v$ and $C$ both involve the $\phi$ transition at every $w\neq
v$, the two entries are the same. On the other hand, the entry of $\rho
(vC_{1}v,C^{\prime})$ corresponding to $v$ is $1$ unless the transition of
$vC_{1}v$ at $v$ (that is, the $\chi$ transition) is the $\phi^{\prime}$
transition, and the $vv$ entry of $\mathcal{I}_{C^{\prime}}(C)$ is $1$ unless
the transition of $C^{\prime}$ at $v$ (the $\phi^{\prime}$ transition) is the
$\chi$ transition. These two entries are also the same, so $\rho
(vC_{1}v,C^{\prime})$ is indeed the same as the row of $\mathcal{I}%
_{C^{\prime}}(C)$ corresponding to $v$.

It follows that $\rho(vC_{1}v,C^{\prime})\cdot\mathcal{I}_{C}(C^{\prime})$ is
the row of $\mathcal{I}_{C^{\prime}}(C)\cdot\mathcal{I}_{C}(C^{\prime})$
corresponding to $v$. We claim that this coincides with the corresponding row
of the identity matrix, i.e., if $w\in V(F)$ and $\kappa_{w}$ is the column of
$\mathcal{I}_{C}(C^{\prime})$ corresponding to $w$ then $\rho(vC_{1}%
v,C^{\prime})\cdot\kappa_{w}=1$ if and only if $w=v$.

Let $P(C,v)$ be the circuit partition that includes $vC_{1}v$ and $vC_{2}v$,
along with all the circuits of $C$ not incident at $v$. Alternatively, we may
describe $P(C,v)$ as the circuit partition that involves $\phi$ transitions at
vertices than $v$, and involves the $\chi$ transition at $v$. As $C$ and
$P(C,v)$ involve the same transitions at vertices other than $v$, the relative
interlacement matrices $\mathcal{I}_{C}(C^{\prime})$ and $\mathcal{I}%
_{P(C,v)}(C^{\prime})$ coincide except for the row and column corresponding to
$v$. Looking at Figure \ref{intext9}, we see that in the cases $\chi
\psi^{\prime}$ and $\chi\chi^{\prime}$, $P(C,v)$ involves the $\phi^{\prime}$
transition at $v$; in the case $\psi\psi^{\prime}$,$\ P(C,v)$ involves the
$\chi^{\prime}$ transition at $v$; and in the case $\psi\chi^{\prime}$,
$P(C,v)$ involves the $\psi^{\prime}$ transition at $v$.

Consider the cases $\chi\psi^{\prime}$ and $\chi\chi^{\prime}$. If $w\neq v$
then the column $\kappa_{w}$ of $\mathcal{I}_{C}(C^{\prime})$ is the same as
the column of $\mathcal{I}_{P(C,v)}(C^{\prime})$ corresponding to $w$, except
for the entry corresponding to $v$; changing this entry does not affect
$\rho(vC_{1}v,C^{\prime})\cdot\kappa_{w}$, because $P(C,v)$ involves the
$\phi^{\prime}$ transition at $v$ and hence the entry of $\rho(vC_{1}%
v,C^{\prime})$ corresponding to $v$ is $0$. Consequently Theorem
\ref{nullspaces} tells us that $\rho(vC_{1}v,C^{\prime})\cdot\kappa_{w}=0$.
Theorem \ref{nullspaces} also tells us that $\rho(vC_{1}v,C^{\prime}%
)\neq\mathbf{0}$; as $\mathcal{I}_{C}(C^{\prime})$ is nonsingular, it follows
that $\rho(vC_{1}v,C^{\prime})\cdot\kappa_{v}=1.$

Now consider the cases $\psi\psi^{\prime}$ and $\chi\chi^{\prime}$. The only
vertex at which $C$ and $P(C,v)$ involve different transitions is $v$, where
one involves the $\chi^{\prime}$ transition and the other involves the
$\psi^{\prime}$ transition; consequently $\mathcal{I}_{P(C,v)}(C^{\prime})$
and $\mathcal{I}_{C}(C^{\prime})$ coincide except for their $vv$ entries,
which are opposites. The entry of $\rho(vC_{1}v,C^{\prime})$ corresponding to
$v$ is $1$, so $\rho(vC_{1}v,C^{\prime})\cdot\mathcal{I}_{P(C,v)}(C^{\prime
})=\mathbf{0}$ tells us that $\rho(vC_{1}v,C^{\prime})\cdot\kappa_{v}\neq0$,
and $\rho(vC_{1}v,C^{\prime})\cdot\kappa_{w}=0$ for $w\neq v$.

The closing comment of the last section applies here too: the special case of
Theorem \ref{intinv} involving compatible Euler systems follows from Jaeger's
proof of Proposition 5 of \cite{J1}, and the general case may be reduced to
the special case by detachment.

\bigskip

\textbf{Acknowledgments}. We would like to thank R. Brijder and D. P. Ilyutko
for many enlightening conversations. The final version of the paper also
profited from the careful attention of two anonymous readers.

\end{document}